\documentclass[12pt]{article}

\def\nfrac#1#2{{\textstyle\frac{#1}{#2}}}

\usepackage{amsmath,amssymb,amsthm}
\usepackage{graphicx}
\usepackage{psfrag}

\newtheorem{theorem}{Theorem}[section]
\newtheorem{lemma}[theorem]{Lemma}
\newtheorem{remark}[theorem]{Remark}

\def\nfrac#1#2{{\textstyle\frac{#1}{#2}}}
\def\dfrac#1#2{\lower0.15ex\hbox{\large$\frac{#1}{#2}$}}

\makeatletter

\renewcommand{\p@enumii}{}

\renewcommand{\p@enumiii}{}
\makeatother

\title{Making Markov chains less lazy}

\author{Catherine Greenhill}

\author{
Catherine Greenhill\\[-0.5ex]
\small School of Mathematics and Statistics\\[-0.5ex]
\small The University of New South Wales\\[-0.5ex]
\small Sydney NSW 2052, Australia\\[-0.5ex]
\small \tt csg@unsw.edu.au}

\date{21 January 2013}

\begin{document}

\maketitle

\begin{abstract}
The mixing time of an ergodic, reversible Markov chain
can be bounded in terms of the eigenvalues of the chain:
specifically, the second-largest eigenvalue and the 
smallest eigenvalue.  It has become standard to focus only
on the second-largest eigenvalue, by making the Markov
chain ``lazy''.  (A lazy chain does nothing at each step with
probability at least $\nfrac{1}{2}$, and has only nonnegative
eigenvalues.)  
 
An alternative approach to bounding the smallest eigenvalue was given by
Diaconis and Stroock~\cite[Proposition 2]{DS91} and
Diaconis and Saloff-Coste~\cite[p.702]{DSC}.
We give examples to show that using this approach
it can be quite easy to obtain a bound on the
smallest eigenvalue of a combinatorial Markov chain which is
several orders of magnitude below the best-known bound on the
second-largest eigenvalue.
\end{abstract}

\section{Introduction}\label{s:intro}

Let $\mathcal{M}$ be an ergodic, reversible Markov chain with
finite state space $\Omega$ and transition matrix $P$.
It is well known that the eigenvalues of $\mathcal{M}$ satisfy 
\[ 1 = \lambda_0 > \lambda_1 \geq \lambda_2 \geq \cdots
           \geq \lambda_{N-1} > -1,\]
where $N=|\Omega|$. 
We refer to $\lambda_{N-1}$ as the \emph{smallest eigenvalue} 
of $\mathcal{M}$.  

The connection between the mixing time of a Markov chain and its
eigenvalues is well-known (see~\cite[Proposition 1]{sinclair}):
\begin{equation}
\label{mix-time} \tau(\varepsilon) \leq (1-\lambda_\ast)^{-1}\, 
   \ln \frac{1}{\epsilon\, \pi_{\min}}
\end{equation}
where $\tau(\varepsilon)$ denotes the mixing time of the Markov chain,
$\pi_{\min} = \min_{x\in \Omega} \pi(x)$ and
\[ \lambda_\ast = \max\{ \lambda_1, \, |\lambda_{N-1}|\}.\]
When studying the mixing time of a Markov chain $\mathcal{M}$
using (\ref{mix-time}),
the approach which has become standard is to make the chain $\mathcal{M}$
\emph{lazy} by replacing $P$ by $(I+P)/2$,
where $I$ denotes the identity matrix.  Then all eigenvalues of the lazy
chain are nonnegative, and only the second-largest eigenvalue must
be investigated.  

A lazy chain can be implemented so that its expected
running time is the same as the mixing time of the original chain.
So the problem with lazy chains is not their efficiency.
In our opinion, the main problem with lazy Markov chains is conceptual:
in order to prove that a Markov chain is fast, we first slow it down.
The device of using lazy Markov chains has been called 
``crude''~\cite[p.~110]{SJ89} and 
``unnatural''~\cite[Chapter 5]{jerrum-book}.

In this note, we aim to advertise an approach for bounding the
smallest eigenvalue of a Markov chain.
This approach was first proposed by Diaconis and Stroock 
in 1991~\cite[Proposition 2]{DS91}, and a modified version was
presented by 
Diaconis and Saloff-Coste two years later~\cite[p.702]{DSC}
(restated as Lemma~\ref{lazy} below).  
The method of~\cite{DSC} has been applied in~\cite{DSC,DS91,goel},
but in the theoretical computer science community it has become common to
work with lazy chains.  We urge researchers to first try 
the approach of~\cite{DSC,DS91}
%Lemma~\ref{lazy}
before choosing to work with a lazy version of their chain.

Finally we remark that in~\cite{directed} the author wrongly claimed
that their~\cite[Lemma 1.3]{directed} was new, when in fact it is precisely
the result of~\cite[p.702]{DSC}. We sincerely apologise for this error.

\subsection{The method}

See~\cite{jerrum-book} for Markov chain definitions not given here.
Write $\mathcal{G}$ for the underlying directed graph of the Markov chain
$\mathcal{M}$, where $\mathcal{G} = (\Omega, \Gamma)$
and each directed edge $e\in \Gamma$ corresponds to a transition of
$\mathcal{M}$.  
If $P(x,x)>0$ then the edge $xx$ is called a \emph{self-loop}
at $x$. Define $Q(e) = Q(x,y) = \pi(x)P(x,y)$ for the edge $e=xy$.
A walk in $\mathcal{G}$ is a sequence of states
$x_0 x_1 \cdots x_\ell$
such that $P(x_j,x_{j+1})> 0$ for $j=0,\ldots, \ell-1$. 
The walk is \emph{closed} if $x_\ell=x_0$.
If a walk has odd length then we call it an \emph{odd walk}.

For each $x\in\Omega$ let $w_x$ be an  odd 
walk from $x$ to $x$ in $\mathcal{G}$.   
(Such a walk exists for each $x$, since the Markov chain
is aperiodic.)  Define
$\mathcal{W} = \{ w_x : x\in\Omega\}$, a set 
of ``canonical closed odd walks''.
For each transition $e\in\Gamma$ and each $w\in\mathcal{W}$,
let $r(e,w)$ denote the number of times that $e$ appears as a
directed edge of $w$.  We can assume that $r(e,w)\leq 2$ for all
transitions $e$ (indeed, if $e$ is a self-loop then we can assume
that $r(e,w)\leq 1$.)  The \emph{congestion} of $\mathcal{W}$,
denoted by $\eta(\mathcal{W})$, is defined by
\[ \eta(\mathcal{W}) = \max_{e\in\Gamma} \, Q(e)^{-1}\, 
           \sum_{x\in\Omega,\, e\in w_x}\, r(e,w_x)\, \pi(x)\, |w_x|.\]

\begin{lemma} \emph{\cite[p.702]{DSC}}\
Suppose that $\mathcal{M}$ is a reversible, ergodic Markov chain with
state space $\Omega$, and let $\mathcal{W}$ be a set of odd walks
defined as above.  Then
\[ (1 + \lambda_{N-1})^{-1} \leq  \frac{\eta(\mathcal{W})}{2}.\]
\label{lazy}
\end{lemma}

If $|w_x|=1$ for all $x\in\Omega$ then
the bound of Lemma~\ref{lazy} simplifies further to 
\begin{equation}
\label{all-loops}
 (1+\lambda_{N-1})^{-1} \leq \dfrac{1}{2}\operatorname{max}_{x\in\Omega} 
P(x,x)^{-1}.
\end{equation}

\begin{remark}
\emph{
Suppose that the graph underlying a Markov chain $\mathcal{M}$
can be obtained from a connected bipartite graph by adding loops
to an exponentially small proportion of states.
For example, many instances
of the \emph{knapsack chain}~\cite{MS99} satisfy this property.
Since every closed odd walk must traverse at least one of
these self-loop edges, it is very difficult to define a set of
canonical closed odd walks with low congestion.
So Lemma~\ref{lazy} is unlikely to be easy to apply in this case.
}
\end{remark}

\section{Applications of the method}\label{s:applications}

We illustrate the use of Lemma~\ref{lazy}
by applying it to three combinatorial Markov chains.
Our applications are all ergodic and reversible with uniform stationary 
distribution, and no edge will be used more than once
in any walk $w_x$ that we define.  In this case the congestion
can be simplified to
\begin{equation}
\label{simple}
 \eta(\mathcal{W}) = \operatorname{max}_{e\in\Gamma} \, P(e)^{-1}
\, \sum_{x\in\Omega,\,\, e\in w_x} \,  |w_x|,
\end{equation}
where $P(e) = P(x,y) = P(y,x)$ for the transition $e=xy$.

\subsection{The switch chain for sampling regular graphs}

Our first application is to the Markov chain for sampling regular
graphs known as the \emph{switch chain}.  
A transition of the chain is performed as follows:
from the current state $G$  (a $d$-regular graph on vertex set $[n]$)
choose an unordered pair of non-incident edges uniformly at random,
let $G'$ be the multigraph obtained from $G$ by deleting these edges
and inserting a perfect matching
of their four endvertices, selected uniformly at random.   If $G'$ has no 
repeated edges then the new state is $G'$, otherwise it is $G$.

The lazy version
of this chain was analysed by
Cooper et al.~\cite{CDG,CDG-corrigendum}.
Clearly $P(G,G)\geq \nfrac{1}{3}$ for every
state $G$ of this chain, so by (\ref{all-loops}) we immediately 
conclude that
\[ (1 + \lambda_{N-1})^{-1}\leq \nfrac{3}{2}.
\]
This is several orders of magnitude smaller than the best-known
bound on $(1-\lambda_1)^{-1}$, which is $O(d^{23} n^8)$ 
(see~\cite{CDG-corrigendum}).

\subsection{Jerrum and Sinclair's matchings chain }

The next application is to the well-known
Markov chain for sampling perfect and near-perfect matchings
of a fixed graph $G$.  A transition of the chain is performed as
follows: from the current state $M$ (which is a
perfect or near-perfect matching of $G$), choose an edge $e\in E(G)$
uniformly at random.  If $M$ is a perfect matching and $e\in M$ then the
new state is $M - \{e\}$.  If $M$ is a near-perfect matching and both
endvertices of $e$ are unmatched in $M$ then the new state is $M \cup \{e\}$.
If $M$ is a near-perfect matching, and exactly one endvertex of $e$
is unmatched in $M$ then let $e'$ be the edge of $M$ which matches the other
endvertex of $e$: the new state is $(M - \{e'\})\cup\{e\}$.
In all other cases the new state is $M$.

The lazy version of this chain was analysed by Jerrum and Sinclair~\cite{JS89,JS96},
If $G$ itself is not a
perfect matching then $P(M,M)\geq 1/|E|$ for all states $M$
of the chain (that is, for all perfect or near-perfect matchings $M$
of $G$). 
Therefore (\ref{all-loops}) implies that
\[ (1 + \lambda_{N-1})^{-1} \leq \frac{|E|}{2}.\]
This bound is at least a factor $n^2$ smaller than the smallest-known
bound on $(1-\lambda_1)^{-1}$, which is $O(n|E|q(n))$
for graphs $G$ for which the ratio between the number of
near-perfect and perfect matchings is $q(n)$ (see~\cite{JS96}). 

\subsection{A heat-bath chain for sampling contingency tables}

Our final application involves contingency tables.
Let $\mathbf{r}=(r_1,\ldots, r_m)$ and $\mathbf{c}=(c_1,\ldots, c_n)$ 
be two vectors of positive 
integers with the same sum.  A \emph{contingency table} with row sums $r$ and 
column sums $c$ is an $m\times n$ matrix $X=(x_{i,j})$ with nonnegative 
integer entries, such that
$\sum_{j=1}^n x_{i,j} = r_i$ for $i=1,\ldots, m$ and 
$\sum_{i=1}^m x_{i,j} = c_j$ for $j=1,\ldots, n$.
Let $\Omega_{\mathbf{r},\mathbf{c}}$ denote the set of all contingency tables 
with row sums $\mathbf{r}$ and column sums $\mathbf{c}$.
To avoid trivialities we assume throughout this section that
$\min\{ m,n\}\geq 2$.

Dyer and Greenhill~\cite{DG-contingency} proposed a Markov chain for sampling
contingency tables, which we will call the \emph{contingency chain}.  
A transition of the chain is performed as follows:  choose a $2\times 2$ 
subsquare of the current table uniformly at random, then  
replace this $2\times 2$ subsquare by a uniformly chosen $2\times 2$
nonnegative integer matrix with the same row and column sums.

The \emph{lazy} contingency
chain does nothing at each step with probability $\nfrac{1}{2}$, and otherwise
performs a transition as described above.
Cryan et al.~\cite{CDGJM} analysed the lazy contingency chain for a constant
number of rows.  They proved that $(1-\lambda_1)^{-1}\leq n^{f(m)}$ for
$m$-rowed contingency tables with $n$ columns, where $m$ is constant
and $f(m)$ is an expression satisfying $f(m)\geq 68m^4$.
We now analyse the smallest eigenvalue of the (non-lazy) contingency chain.

There is always a positive probability that the next state $X'$ of the 
contingency chain is equal to the current state $X$,  
since the heat-bath step may simply replace the
chosen $2\times 2$ subsquare with its current contents.  
However, 
the minimum of $P(X,X)$ over all states $X$ depends on
$\mathbf{r}$ and $\mathbf{c}$.
(To see this, consider $2\times 2$ squares.)
We prefer a bound which depends only on $m$ and $d$, and so we do not
simply apply (\ref{all-loops}).

\begin{lemma}
\label{non-lazy-contingency}
Let $\mathbf{r}=(r_1,\ldots, r_m)$ and $\mathbf{c}=(c_1,\ldots, c_n)$ 
be vectors of positive 
integers with a common sum which satisfy 
\[ r_1\geq r_2\geq \cdots \geq r_m\quad \text{ and } \quad
   c_1\geq c_2 \geq \cdots \geq c_n.\]
Suppose that $\min\{r_1,\, c_1\}\geq 2$ and $\max\{ m,n\}\geq 3$.
The smallest eigenvalue of the contingency chain on 
$\Omega_{\mathbf{r},\mathbf{c}}$ satisfies
\[ (1+\lambda_{N-1})^{-1} \leq 45\, m^3n^3. \]
\end{lemma}

\begin{proof}
Write $[a] = \{ 1,2,\ldots, a\}$ for $a\in\mathbb{Z}^+$.
From $X = (x_{i,j})\in\Omega_{\mathbf{r},\mathbf{c}}$, 
first suppose that there exists a 
5-tuple $(i_1,i_2,i_3,j_1,j_2)$ such
that 
\begin{itemize}
\item $i_1,i_2,i_3$ are distinct elements of $[m]$, 
\item $j_1,j_2$ are distinct elements of $[n]$,
\item $x_{i_1,j_1},\, x_{i_2,j_1},\, x_{i_3,j_2}$ are all positive.
\end{itemize}
Then $(i_1,i_2,i_3,j_1,j_2)$ is called \emph{row-good for $X$}, and $X$ is called \emph{row-good}.
If $X$ is row-good, fix the lexicographically least 5-tuple $(i_1,i_2,i_3,j_1,j_2)$ which is row-good for
$X$ and consider the following sequence of three transitions on the $3\times 2$ subsquare defined
by rows $i_1,i_2,i_3$ and columns $j_1,j_2$:
\[ \begin{pmatrix}  y_{1,1} & y_{1,2}\\  y_{2,1} & y_{2,2}\\ y_{3,1} & y_{3,2}\end{pmatrix} \,\, \Longrightarrow
  \,\, \begin{pmatrix}  y_{1,1}-1 & y_{1,2}+1\\  y_{2,1} & y_{2,2}\\ y_{3,1}+1 & y_{3,2}-1\end{pmatrix} \,\, \Longrightarrow
\,\,  \begin{pmatrix}  y_{1,1} & y_{1,2}\\  y_{2,1} -1 & y_{2,2}+1\\ y_{3,1}+1 & y_{3,2}-1\end{pmatrix} \,\, \Longrightarrow
\begin{pmatrix}  y_{1,1} & y_{1,2}\\  y_{2,1} & y_{2,2}\\ y_{3,1} & y_{3,2}\end{pmatrix}.
\]
(For notational convenience we have written $y_{k,\ell}$ for $x_{i_k,j_\ell}$ in the above.)
Note that all intermediate matrices are nonnegative, due to the row-good property.
This defines a walk $w_X$ of length 3 from $X$ to $X$ in the graph underlying the contingency chain. 

We can define 5-tuples $(i_1,i_2,j_1,j_2,j_3)$ which are \emph{column-good for $X$}
in the analogous way, and say that
$X$ is column-good if there is a 5-tuple which is column-good for $X$.   If $X$ is column-good then
taking the transpose of each matrix in the sequence of transitions above defines 
an odd walk $w_X$ of length 3 from $X$ to $X$.  

Finally, suppose that $X\in\Omega_{\mathbf{r},\mathbf{c}}$ is not row-good and 
is not column-good.
Such an $X$ is said to be \emph{bad}.
Then no row or column of $X$ contains more than one positive entry.   Since all row and column
sums are positive, it follows that $m=n\geq 3$ and that every row and column contains exactly one
positive entry.  
Let $(i_1,i_2,i_3,j_1,j_2,j_3)$ be the lexicographically-least 6-tuple such that
\begin{itemize}
\item $i_1,i_2,i_3$ are distinct elements of $[m]$, 
\item $j_1,j_2,j_3$ are distinct elements of $[n]$, 
\item $x_{i_1,j_1} \geq 2$, while  $x_{i_2,j_2}$ and $x_{i_3,j_3}$ are positive.
\end{itemize}
(The conditions on $\mathbf{r}$ and $\mathbf{c}$ guarantee that such a 6-tuple exists.)  Consider the following sequence
of 5 transitions, performed on the $3\times 3$ subsquare defined by rows $i_1,i_2,i_3$ and columns
$j_1,j_2,j_3$:
\begin{align*}
\begin{pmatrix} y_{1,1} & 0 & 0\\ 0 & y_{2,2} & 0\\ 0 & 0 & y_{3,3}\end{pmatrix}
  & \Longrightarrow \,\,
  \begin{pmatrix} y_{1,1}-1 & 1 & 0\\ 1 & y_{2,2}-1 &0\\ 0 & 0 & y_{3,3}\end{pmatrix} \,\, \Longrightarrow \,\,
  \begin{pmatrix} y_{1,1}-1 & 1 & 0\\ 0 & y_{2,2}-1 &1\\ 1 & 0 & y_{3,3}-1\end{pmatrix} \\
   & \Longrightarrow \,\,
    \begin{pmatrix} y_{1,1}-2 & 1 & 1\\ 1 & y_{2,2}-1 &0\\ 1 & 0 & y_{3,3}-1\end{pmatrix} \,\, \Longrightarrow \,\,
    \begin{pmatrix} y_{1,1}-1& 0 & 1\\ 0 & y_{2,2} &0\\ 1 & 0 & y_{3,3}-1\end{pmatrix} \\
      & \Longrightarrow \,\,
    \begin{pmatrix} y_{1,1}& 0 & 0\\ 0 & y_{2,2} &0\\ 0 & 0 & y_{3,3}\end{pmatrix}.
\end{align*}
This defines a walk $w_X$ of length 5 from $X$ to $X$ in the graph underlying the chain.

Now we must analyse the set 
$\mathcal{W} = \{ w_X : X\in\Omega_{\mathbf{r},\mathbf{c}}\}$ of odd walks 
defined above.
Let $e=(Z,Z')$ be a transition of the contingency chain.  
Then $Z$ and $Z'$ only differ in a $2\times 2$
subsquare defined by rows $i,i'$ and columns $j,j'$.    

First we seek row-good $X$ with $e\in w_X$.  Let $i''\not\in\{i,i'\}$ 
be another row index, and fix one of 
the 6 ways to arrange $(i,i',i'',j,j')$ as $(i_1,i_2,i_3,j_1,j_2)$.  This gives enough information to uniquely
identify a potential candidate for $X$.   For example, if the transition $e$ involves rows $i_1$ and $i_3$ then 
$X=Z$, while if the transition $e$ involves rows $i_2$ and $i_3$ then $X=Z'$.  If $e$ involves rows $i_1$ and $i_2$
then $e$ is the second transition in the sequence, and $X$ can be obtained from $Z$ by reversing
the first transition in the sequence: namely, adding 1 to entries $(i_1,j_1)$ and $(i_3,j_2)$ and
subtracting 1 from entries $(i_1,j_2)$ and $(i_3,j_1)$.  If $X$ is a valid contingency table then
$(i_1,i_2,i_3,j_1,j_2)$ is row-good for $X$.  If it is the lexicographically
least such 5-tuple for $X$ then $e\in w_X$.
This identifies at most $12(m-2)$ tables $X$ such that $e\in w_X$.  
(This is an overcount, but good enough for our purposes.)

By choosing a third column index $j''\not\in\{j,j'\}$, an analogous argument shows that there are at most
$12(n-2)$ column-good tables $X$ with $e\in w_X$.

Finally, we seek bad tables $X$  such that $e\in w_X$.
Choose a row index $i''\not\in\{ i,i'\}$ and a column index 
$j''\not\in\{j,j'\}$, and fix one of the at most
36 ways to arrange $(i,i',i'',j,j',j'')$ as $(i_1,i_2,i_3,j_1,j_2,j_3)$.  Now
each transition in the sequence alters a different $2\times 2$ subsquare except the first and fourth,
which both alter rows $i_1,i_2$ and columns $j_1,j_2$.  Hence, arguing as above, there are at
most two choices for $X$, for each fixed 6-tuple.   This gives at most $72(m-2)(n-2)$ bad tables
$X$ such that $e\in w_X$.

Combining all this, we find that the congestion parameter $\eta(\mathcal{W})$ satisfies
\[ \eta(\mathcal{W}) \leq \binom{m}{2}\,\binom{n}{2}\, \left(36(m-2)+36(n-2)+360(m-2)(n-2)\right)
    \leq 90\, m^3n^3,
    \]
and applying Lemma~\ref{lazy} completes the proof.    
\end{proof}

Again we observe that this bound on $(1+\lambda_{N-1})^{-1}$ is several orders
of magnitude lower than the best-known
bound on the second-largest eigenvalue~\cite{CDGJM}. 

\begin{remark}
It has recently been shown~\cite{GU} that the contingency chain described above has no negative eigenvalues.
We include Lemma~\ref{non-lazy-contingency} here to illustrate an application of Lemma~\ref{lazy}
involving walks of length greater than one.
\end{remark}


\begin{thebibliography}{99}

\bibitem{CDG}
C.~Cooper, M.E.~Dyer and C.~Greenhill,
Sampling regular graphs and a peer-to-peer network,
\emph{Combinatorics, Probability and Computing} {\bf 16}
(2007), 557--593.

\bibitem{CDG-corrigendum}
C.~Cooper, M.E.~Dyer and C.~Greenhill,
Corrigendum: Sampling regular graphs and a peer-to-peer network.
{\tt arXiv:1203.6111v1 [math.CO]}

\bibitem{CDGJM}
M.~Cryan, M.~Dyer, L.A.~Goldberg, M.~Jerrum and R.~Martin,
Rapidly mixing Markov chains for sampling contingency tables with
a constant number of rows,
\emph{SIAM Journal on Computing } {\bf 36} (2006), 247--278.

\bibitem{DSC}
P.~Diaconis and L.~Saloff-Coste,
Comparison theorems for reversible Markov chains,
\emph{Annals of Applied Probability} {\bf 3} (1993), 696--730.

\bibitem{DS91}
P.~Diaconis and D.~Stroock,
Geometric bounds for eigenvalues of Markov chains,
\emph{Annals of Applied Probability} {\bf 1} (1991),
36--61.


%\bibitem{DGJM}
%M.~Dyer, L.A.~Goldberg, M.~Jerrum and R.~Martin,
%Markov chain comparison,
%\emph{Probability Surveys} {\bf 3} (2006), 89--111.

\bibitem{DG-contingency}
M.E.~Dyer and C.~Greenhill,
Polynomial-time counting and sampling of two-rowed contingency tables,
\emph{Theoretical Computer Science} {\bf 246} (2000), 265--278.

\bibitem{goel}
S.~Goel, Analysis of top to bottom-k shuffles, 
\emph{Annals of Applied Probability} {\bf 16} (2006), 30--55.

\bibitem{directed}
C.~Greenhill,
A polynomial bound on the mixing time of a Markov chain for 
sampling regular directed graphs, 
\emph{Electronic Journal of Combinatorics} {\bf 18} (2011), \#P234.

\bibitem{GU}
C.~Greenhill and M.~Ullrich,
Heat-bath Markov chains have no negative eigenvalues
(preprint, 2013).\
{\tt arxiv.org:1301.4055 [math.CO]}

\bibitem{jerrum-book}
M.~Jerrum,
\emph{Counting, Sampling and Integrating: algorithms and 
complexity}, 
Lectures in Mathematics -- ETH Z{\" u}rich, Birkh{\" a}user, Basel, 2003. 

\bibitem{JS89}
M.~Jerrum and A.~Sinclair,
Approximating the permanent,
\emph{SIAM Journal on Computing} {\bf 18} (1989), 
1149-1178.

\bibitem{JS96}
M.~Jerrum and A.~Sinclair,
The Markov chain Monte Carlo method: an approach to approximate counting 
and integration,
in \emph{Approximation Algorithms for NP-hard Problems} 
(Dorit Hochbaum, ed.), PWS, 1996.

\bibitem{MS99}
B.~Morris and A.~Sinclair,
Random walks on truncated cubes and sampling 0-1 knapsack
solutions, in 
\emph{Proceedings of the 40th IEEE Symposium on Foundations
of Computer Science}, IEEE Computer Society Press, 1999,
pp.~230--240.

\bibitem{sinclair}
A.~Sinclair, 
Improved bounds for mixing rates of Markov chains
and multicommodity flow, 
\emph{Combinatorics, Probability and Computing}
{\bf 1} (1992), 351--370.

\bibitem{SJ89}
A.~Sinclair and M.~Jerrum,
Approximate counting, uniform generation and rapidly
mixing Markov chains, 
\emph{Information and Computation} {\bf 82} (1989), 93--133.

\end{thebibliography}
\end{document}